\documentclass[11pt]{amsart}
\usepackage[numbers,sort&compress]{natbib}
\bibpunct{[}{]}{,}{n}{}{;}
\usepackage{amsmath,amsthm,amsopn,amstext,amscd,amsfonts,amssymb}
\usepackage{comment}

\theoremstyle{plain}

\newcommand{\PP}{\mathbb{P}}
\newcommand{\RR}{\mathbb{R}}

\newtheorem{definition}{\sc Definition}[section]
\newtheorem{teo}{\sc Theorem}[section]
\newtheorem{prop}{\sc Proposition}[section]

\newtheorem{obs}{\sc Remark}[section]

\newcommand{\supp}[1]{\mbox{{\rm supp\/}}#1}

\renewcommand{\a}{\alpha}
\renewcommand{\b}{\beta}

\renewcommand{\d}{\delta}
\newcommand{\g}{\gamma}

\renewcommand{\l}{\lambda}
\renewcommand{\L}{\Lambda}
\renewcommand{\O}{\Omega}
\newcommand{\D}{\Delta}
\newcommand{\s}{\sigma}

\begin{document}

\title[Markov-type inequalities and duality in weighted Sobolev spaces]{Markov-type inequalities and duality in weighted Sobolev spaces}

\author[F. Marcell\'an]{Francisco Marcell\'an$^{(1)}$}
\address{Instituto de Ciencias Matem\'aticas (ICMAT),
Campus de Cantoblanco, Universidad Aut\'onoma de Madrid, Calle Nicol\'as Cabrera, no. 13-15, 28049 Madrid, Spain. \\
Departamento de Matem\'aticas, Universidad Carlos III de Madrid, Avenida de la Universidad 30, 28911 Legan\'es, Madrid,
Spain} \email{pacomarc@ing.uc3m.es}
\thanks{$(1)\,\,\,$ Supported in part by a
grant from Ministerio de Econom{\'\i}a y Competitividad (MTM2012-36732-C03-01), Spain.}

\author[Y. Quintana]{Yamilet Quintana$^{(1)}$}
\address{Departamento de Matem\'aticas Puras y Aplicadas,
Edificio Matem\'aticas y Sistemas (MYS), Apartado Postal: 89000, Caracas 1080 A, Universidad Sim\'on Bol\'{\i}var,
Venezuela} \email{yquintana@usb.ve}

\author[J.M. Rodr{\'\i}guez]{Jos\'e M. Rodr{\'\i}guez$^{(2)}$}
\address{Departamento de Matem\'aticas, Universidad Carlos III de Madrid,
Avenida de la Universidad 30, 28911 Legan\'es, Madrid, Spain} \email{jomaro@math.uc3m.es}
\thanks{$(2)\,\,\,$ Supported in part by a
grant from Ministerio de Econom{\'\i}a y Competitividad (MTM2013-46374-P), Spain.}

\date{\today. }

\begin{abstract}
The aim of this paper is to provide Markov-type inequalities in the setting of weighted Sobolev spaces when the
considered weights are generalized classical weights. Also, as results of independent interest, some basic facts about Sobolev spaces with respect to certain vector measures are stated.
\\

\hspace{-.6cm} {\it Key words and phrases:} Extremal problems; Markov-type inequality; weighted Sobolev norm; weighted
$L^{2}$-norm; duality.\\

\end{abstract}

\maketitle{}

\markboth{F. MARCELL\'AN, Y. QUINTANA AND J.M. RODR\'IGUEZ \mbox{} }{Markov-type inequalities in weighted Sobolev
spaces}

\section{Introduction}

\label{[Section-1]-Introduction}

Given a norm on the linear space $\PP$ of polynomials with real coefficients, the so-called Markov-type inequalities are estimates connecting the norm of derivatives of a polynomial with the norm of the polynomial itself. These inequalities are interesting by themselves and play a  fundamental role in the proof of many inverse theorems in polynomial approximation theory (cf. \cite{milo,PQ2011} and the references therein).\\

It is well known that for every polynomial $P$ of degree at most $n$, the Markov inequality
$$\|P^{\prime}\|_{L^{\infty}([-1,1])}\leq n^{2} \|P\|_{L^{\infty}([-1,1])}$$
holds and it is optimal since you have equality for the Chebyshev polynomials of the first kind.\\

The above inequality has been extended to the $p$ norm ($p\geq 1$) in \cite{HST}. For every polynomial $P$ of degree at most $n$ their result reads
$$\|P^{\prime}\|_{L^{p}([-1,1])}\leq C(n,p) n^{2} \|P\|_{L^{p}([-1,1])}$$
where the value of $C(n,p)$ is explicitly given in terms of $p$ and $n$. Indeed, you have a bound  $C(n,p) \leq 6e ^{1+ 1/e}$ for $n>0$ and $p\geq 1$. In \cite{Goe} admissible values for $C(n,p)$ and some computational results for $p=2$ are deduced. Notice that for any $p>1$ and every polynomial $P$ of degree at most $n$
$$\|P^{\prime}\|_{L^{p}([-1,1])}\leq C n^{2} \|P\|_{L^{p}([-1,1])},$$
where $C$ is explicitly given and it is less than the constant $C(n,p)$  in \cite{HST}.\\

On the other hand, using matrix analysis, in \cite{Dor} it is proved that the exact value of $C(n,2)$ is the greatest singular value of the matrix $A_{n}=[a_{j,k}]_{0\leq j\leq n-1, 0\leq k\leq n}$, where $a_{j,k}= \int_{-1}^{1} p'_{j}(x) p_{k} (x) dx$ and $\{p_{n}\}_{n=0}^{\infty}$ is the sequence of orthonormal Legendre polynomials. A simple proof of this result, with an interpretation of the the sharp constant $C(n,2)$ as the largest positive zero of a polynomial as well as an explicit expression of the extremal polynomial (the polynomial such that the inequality becomes an equality) in the $L^{2}$- Markov inequality appears in \cite{Kroo}.\\

If you consider weighted $L^{2}$-spaces, the problem becomes more difficult. For instance, let $\|\,\cdot\,\|_{L^{2}((a,b), w)}$ be a weighted $L^{2}$-norm on $\PP$, given by
$$\|P\|_{L^{2}((a,b), w)}=\left(\int_{a}^{b}|P(x)|^{2} w(x)dx\right)^{1/2},$$
where $w$ is an integrable function on $(a, b)$, $-\infty \leq a<b\leq \infty$, such that $w> 0$ a.e. on $(a,b)$ and all
moments
$$r_{n}:= \int_{a}^{b}x^{n}w(x)dx, \qquad n\ge 0,$$
are finite.
It is clear that there exists a constant $\gamma_{n}= \gamma_{n}(a,b,w)$ such that
\begin{equation}
\label{mark5} \|P^{\prime}\|_{L^{2}((a,b), w)}\leq \gamma_{n}\|P\|_{L^{2}((a,b), w)}, \mbox{ for all } P\in \PP_{n},
\end{equation}
where $\PP_{n}$ is the space of polynomials with real coefficients of degree at most $n$. Indeed, the sharp constant is the greatest singular value of the matrix $B_{n}= [b_{j,k}]_{0\leq j\leq n-1,0\leq k\leq n}$, where $b_{j,k}= \int_{-1}^{1} p'_{j}(x) p_{k} (x) w(x)dx$ and $\{p_{n}\}_{n=0}^{\infty}$ is the orthonormal polynomial system with respect to the positive measure $w(x)dx$. Thus, from a computational point of view you need to find the connection coefficients between the sequences  $\{p'_{n}\}_{n=0}^{\infty}$ and $\{p_{n}\}_{n=0}^{\infty}$ in order to proceed with the computation of the matrix, and in a second step, to find the greatest singular value of the matrix $B_{n}.$ Notice that for classical weights (Jacobi, Laguerre and Hermite), such connection coefficients can be found in a simple way.\\

Mirsky \cite{mirsky} showed that the best
constant $\gamma_{n}^{*}:= \sup_{P\in \PP_{n}}\{||P^{\prime}||_{L^{2}((a,b), w)}:||P||_{L^{2}((a,b), w)}=1 \} $ in (\ref{mark5}) satisfies
\begin{equation}
\label{mark6} \gamma_{n}^{*} \leq
 \left(\sum_{\nu=1}^{n}\nu\|p_{\nu}^{\prime}\|^{2}_{L^{2}((a,b),
w)}\right)^{\!\!1/2}.
\end{equation}

Notice that the main interest of the above result is however qualitative, since the bound specified by $(\ref{mark6})$ can
be very  crude. In fact, when $w(x)= e^{-x^{2}}$ on $(-\infty,\infty)$, the estimate $(\ref{mark6})$ becomes
$$
\gamma_{n}^{*}\leq
 \left(\sum_{\nu=1}^{n}2\nu^{2}\right)^{\!\!1/2}=\sqrt{\frac{1}{3}n(n+1)(2n+1)}=O\big(n^{3/2}\big).
 $$
The contrast between this estimate and the classic result of Schmidt \cite{schmidt}, which establishes  $
\gamma_{n}^{*}=\sqrt{2n}$,  is evident.\\


Also, when we consider  the weighted $L^{2}$-norm associated with  the Laguerre weight $w(x):= x^\a e^{-x}$ in $[0,\infty)$, the
results in \cite{ADK} give the following inequality
\begin{equation}
\label{mark7} \|P'\|_{L^2(w)} \le C_\a n\, \|P\|_{L^2(w)}, \mbox{ for all } P\in \PP_{n}.
\end{equation}
\\
Notice that the nature of the extremal problems associated to the inequalities (\ref{mark5}) and (\ref{mark7}) is
different, since in the first case the constant on the right-hand side of  (\ref{mark5}) depends on $n$, while in the
second one the multiplicative constant $C_\a$ on the right-hand side of  (\ref{mark7})  is independent of $n$.\\

There exist a lot of results on Markov-type inequalities (see, e.g. \cite{DK,DMS,milo}, and the references therein). In connection with the research in the field of the weighted approximation by polynomials, Markov-type inequalities have been proved for various weights, norms, sets over which the norm is taken (see, e.g. \cite{nai} and the references therein) and more recently, the study of asymptotic behavior of the sharp constant involved in some kind of these inequalities have been done in \cite{ADK} for Hermite, Laguerre and Gegenbauer weights and in \cite{ADK2} for Jacobi weights with parameters satisfying some constraints.\\

On the other hand, a similar problem connected with the Markov-Bernstein inequality has been analyzed in \cite{GM} when you try to determine the sharp constant $C(n,m;w)$ such that
\begin{equation}
\|A ^{m/2} P^{(m)}\|_{L^{2}((a,b), w)}\leq C(n,m;w) \|P\|_{L^{2}((a,b), w)}, \mbox{ for all } P\in \PP_{n}.
\end{equation}

Here $w$ is a classical weight satisfying a Pearson equation $(A(x)w(x))'= B(x) w(x)$ and $A, B$  are polynomials of degree at most 2 and 1, respectively.\\

An analogue of the Markov-Bernstein inequality for linear operators $T$ from $\PP_{n}$ into  $\PP$ has been studied in \cite{kwon} in terms of singular values of matrices. Some illustrative examples when $T$ is either the derivative or the difference operator and you deal with some classical weights (Laguerre, Gegenbauer in the first case, Charlier, Meixner in the second one) are shown. Another recent application of Markov-Bernstein-type inequalities can be found in \cite{bun}.

\medskip

With these ideas in mind, one of the authors of the present paper  posed in 2008 during a conference on Constructive
Theory of Functions held in Campos do Jord\~{a}o, Brazil, the following problem: Find the analogous of Markov-type
inequalities in the setting of weighted Sobolev spaces. A partial answer of this problem
was given in \cite{PQ2011}, considering an extremal problem with similar conditions to those given by Mirsky, and
following the scheme of Kwon and Lee \cite{kwon}, mainly.

\medskip

The first part of this paper is devoted to  provide another partial solution of the above problem, which is based on the adequate use of inequalities of kind  (\ref{mark7}) \cite{ADK,DK,schmidt}, in the setting of weighted Sobolev spaces, when the considered weights are generalized classical weights.
In the second part we study some basic facts about Sobolev spaces with respect to measures:
separability, reflexivity, uniform convexity and duality, which to the best of our knowledge are not available in the current literature.
These Sobolev spaces appear in a natural way and are a very useful tool when we study the asymptotic behavior of Sobolev orthogonal polynomials
(see \cite{CPRR}, \cite{LPP}, \cite{LP}, \cite{R}, \cite{R2}, \cite{R3}, \cite{RARP2}, \cite{RS}).

\medskip

The outline of the paper is as follows. The first part of Section \ref{[Section-2]-Stype} provides some short
background about Markov-type inequalities in $L^2$ spaces with classical weights and the second one deals with a
Markov-type inequality corresponding to each weighted Sobolev norm with respect to these classical weights and to some
generalized weights (see Theorem \ref{Main_result}).
Section \ref{[Section-3]-Basic} contains definitions and a discussion about the
appropriate vector measures which we will need in order to get completeness of our Sobolev spaces with respect to
measures.
Finally, Section \ref{[Section-4]-Dual} contains some basic results on Sobolev spaces with
respect to the vector measures defined in the previous section (see Theorems \ref{t:reflexive} and \ref{t:dual}):
separability, reflexivity, uniform convexity and duality.

\section{Markov-type inequalities in Sobolev spaces with weights}
\label{[Section-2]-Stype}

The following proposition summarizes the Markov-type inequalities in $L^2$ spaces with classical weights, which will be  used in the sequel.
Recall that we denote by $\PP_n$ the linear space of polynomials with real coefficients and degree less than or equal to $n$.

\medskip

\begin{prop}
\label{prop1}
The following inequalities are satisfied.

\begin{enumerate}
\item Laguerre case \cite{ADK}:
$$
\|P'\|_{L^2(w)} \le C_\a n\, \|P\|_{L^2(w)},
$$
where $w(x):= x^\a e^{-x}$ in $[0,\infty)$, $\alpha>-1$ and  $P \in \PP_n$.

\item Generalized Hermite case \cite{schmidt}, \cite{DK}:
$$
\|P'\|_{L^2(w)} \le \sqrt{2n}\; \|P\|_{L^2(w)},
$$
where $w(x):= |x|^{\a} e^{-x^2}$ in $\RR$, $\a \ge 0$ and $P \in \PP_n$.

\item Jacobi case \cite{schmidt} (see also \cite{DMS}):
$$
\|P'\|_{L^2(w)} \le C_{\a,\b}\, n^2\, \|P\|_{L^2(w)},
$$
where $w(x):= (1-x)^{\a} (1+x)^{\b}$ in $[-1,1]$, $\alpha,\beta>-1$  and $P \in \PP_n$.
\end{enumerate}

\medskip
The multiplicative constants $C_{\a}$ and $C_{\a,\b} $ are independent of $n$.
\end{prop}

In Theorem \ref{Main_result} below we extend these results to the context of weighted Sobolev spaces.
We want to remark that the proof provides explicit expressions for the involved constants.

\begin{teo}
\label{Main_result}
The following inequalities are satisfied.

\begin{enumerate}
\item Laguerre-Sobolev case:
$$
\|P'\|_{W^{k,2}(w,\l_1w,\dots,\l_k w)} \le C_\a n\, \|P\|_{W^{k,2}(w,\l_1w,\dots,\l_k w)},
$$
where $w(x):= x^\a e^{-x}$ in $[0,\infty)$, $\alpha>-1$, $\l_1,\dots,\l_k \ge 0$, $P \in \PP_n$ and $C_\a$ is the same constant as in Proposition \ref{prop1} $(1)$.

\item Generalized Hermite-Sobolev case:
$$
\|P'\|_{W^{k,2}(w,\l_1w,\dots,\l_k w)} \le \sqrt{2n}\, \|P\|_{W^{k,2}(w,\l_1w,\dots,\l_k w)},
$$
where  $w(x):= |x|^{\a} e^{-x^2}$ in $\RR$, $\a\ge 0$, $\l_1,\dots,\l_k \ge 0$ and $P \in \PP_n$.

\item Jacobi-Sobolev case:
$$
\|P'\|_{W^{k,2}(w,\l_1w,\dots,\l_k w)} \le C_{\a,\b}\, n^2\, \|P\|_{W^{k,2}(w,\l_1w,\dots,\l_k w)},
$$
where $w(x):= (1-x)^{\a} (1+x)^{\b}$ in $[-1,1]$, $\a,\b>-1$, $\l_1,\dots,\l_k \ge 0$, $P \in \PP_n$ and $C_{\a,\b}$ is the constant in Proposition \ref{prop1} $(3)$.

\item Let us consider the generalized Jacobi weight $w(x):= h(x) \Pi_{j=1}^{r} |x-c_j|^{\g_j}$ in $[a,b]$ with $c_1,\dots , c_r \in \RR$,
$\g_1, \dots , \g_r\in \RR$, $\g_j > -1$ when $c_j \in [a,b]$,
and $h$ a measurable function satisfying
$0<m\le h \le M$ in $[a,b]$ for some constants $m,M$.
Then we have
$$
\|P'\|_{W^{k,2}(w,\l_1w,\dots,\l_k w)} \le
 C_1(a,b,c_1, \dots , c_{r}, \g_1, \dots , \g_{r},m,M)\, n^2\, \|P\|_{W^{k,2}(w,\l_1w,\dots,\l_k w)},
$$
for every $\l_1,\dots,\l_k \ge 0$ and $P \in \PP_n$.

\item Consider now the generalized Laguerre weight $w(x):= h(x) \Pi_{j=1}^{r} |x-c_j|^{\g_j} e^{-x}$ in $[0,\infty)$ with
$c_1<\dots < c_r$,
$c_r \ge 0$, $\g_1, \dots , \g_r\in \RR$, $\g_j > -1$ when $c_j \ge 0$,
and $h$ a measurable function satisfying $0<m\le h \le M$ in $[0,\infty)$ for
some constants $m,M.$

$(5.1)$ If $\sum_{j=1}^{r-1} \g_j=0$, then
$$
\|P'\|_{W^{k,2}(w,\l_1w,\dots,\l_k w)} \le
 C_{2}(c_1, \dots , c_{r}, \g_1, \dots , \g_{r},m,M)\, n^2\, \|P\|_{W^{k,2}(w,\l_1w,\dots,\l_k w)},
$$
for every $\l_1,\dots,\l_k \ge 0$ and $P \in \PP_n$.

$(5.2)$ Assume that $c_1<\cdots < c_r$ and $\sum_{j=1}^{r} \g_j > -1$. Let $r_0:=\min \{1\le j \le r\,|\, c_j\ge 0 \}$,
$\g_{r_0-1}':=\g_{r+1}':=0$
and $\g_{j}':=\g_{j}$ for every $r_0 \le j \le r$.
Assume that $\max\{\g_j',\g_{j+1}'\} \ge -1/2$ for every $r_0-1 \le j \le r$.
Then we have
$$
\|P'\|_{W^{k,2}(w,\l_1w,\dots,\l_k w)} \le
 C_{2}'(c_1, \dots , c_{r}, \g_1, \dots , \g_{r},m,M)\, n^{a'}\, \|P\|_{W^{k,2}(w,\l_1w,\dots,\l_k w)},
$$
for every $\l_1,\dots,\l_k \ge 0$ and $P \in \PP_n$, where
$$
a':= \max \Big\{ 2\,,\,\frac{b'+2}{2}\, \Big\} ,
\quad
b':= \max_{r_0-1\le j \le r} \big(\g_j'+\g_{j+1}'+|\g_j'-\g_{j+1}'| + 2 \big) .$$

\item Let us consider the generalized Hermite weight $w(x):= h(x) \Pi_{j=1}^{r} |x-c_j|^{\g_j} e^{-x^2}$ in $\RR$ with $c_1<\cdots < c_r$, $\g_1, \dots , \g_r > -1$
with $\sum_{j=1}^r \g_j \ge 0$ and $h$ a measurable function satisfying $0<m\le h \le M$ in $\RR$ for some constants
$m,M.$
Define $\g_0:=\g_{r+1}:=0$ and assume that $\max\{\g_j,\g_{j+1}\} \ge -1/2$ for every $0 \le j \le r$.
Then we have
$$
\|P'\|_{W^{k,2}(w,\l_1w,\dots,\l_k w)} \le
 C_3(c_1, \dots , c_{r}, \g_1, \dots , \g_{r},m,M)\, n^a\, \|P\|_{W^{k,2}(w,\l_1w,\dots,\l_k w)},
$$
for every $\l_1,\dots,\l_k \ge 0$ and $P \in \PP_n$, where
$$
a:= \max \Big\{ 2\,,\,\frac{b+1}{2}\, \Big\} ,
\quad
b:= \max_{0\le j \le r} \big(\g_j+\g_{j+1}+|\g_j-\g_{j+1}| + 2 \big) .
$$
\end{enumerate}

\medskip
In each case the multiplicative constants depend just on the specified parameters (in particular, they do not depend on $n$).
\end{teo}

\begin{obs}
Note that $(4)$, $(5),$ and $(6)$ are new results in the classical (non-Sobolev) context
(taking $\l_1=\dots=\l_k = 0$).
In $(5.2)$, there is no hypothesis on $\sum_{j=1}^{r-1} \g_j$.
\end{obs}

\begin{proof}
First of all, note that if the inequality
$$
\|P'\|_{L^2(w)} \le C(n,w) \, \|P\|_{L^2(w)}
$$
holds for every polynomial $P \in \PP_n$ and some fixed weight $w$, then we have
\begin{equation}
\label{ec1}
\|P^{(j+1)}\|_{L^2(\l w)} \le C(n,w) \, \|P^{(j)}\|_{L^2(\l w)}
\end{equation}
for every polynomial $P \in \PP_n$ and every $\l \ge 0$.
Consequently, for the weighted Sobolev norm on $\PP$
$$\|P\|_{W^{k,2}(w,\l_1w,\dots,\l_k w)}:=\Big(\big\|P\big\|^2_ {L^2( w)}+ \sum_{j=1}^k
\big\|P^{(j)}\big\|^2_ {L^2(\l_j w)} \Big)^{1/2}, \quad
\l_1,\dots,\l_k \ge 0,$$
we have
\begin{equation}
\label{ec2}
\|P'\|_{W^{k,2}(w,\l_1w,\dots,\l_kw)} \le C(n,w) \, \|P\|_{W^{k,2}(w,\l_1w,\dots,\l_kw)},
\end{equation}
for every polynomial $P \in \PP_n$ and every $\l_1,\dots,\l_k \ge 0$.

Thus, $(1)$, $(2)$ and $(3)$ hold.

\medskip

In order to prove $(4)$, note that
using an affine transformation of the form $Tx=\a_1x+\a_2$, we obtain from Proposition \ref{prop1} $(3)$
$$
\|P'\|_{L^2(w)} \le C(a_1,a_2,\a,\b)\, n^2\, \|P\|_{L^2(w)},
$$
for the weight $w(x):= (a_2-x)^{\a} (x-a_1)^{\b}$ in $[a_1,a_2]$ and every polynomial $P \in \PP_n$.

\smallskip

Without loss of generality we can assume that $a\le c_1<\dots < c_{r} \le b$, since otherwise
we can consider
$$
w(x)= \tilde{h}(x) \prod_{\substack{ 1\le j\le r \\ c_j\in [a,b]}}  |x-c_j|^{\g_j} ,
\qquad
\tilde{h}(x):= h(x) \prod_{\substack{ 1\le j\le r \\ c_j\notin [a,b]}} |x-c_j|^{\g_j} .
$$

If we define $c_0:=a$, $c_{r+1}:=b$ and $\g_0:=\g_{r+1}:=0$, then we can write $w(x)= h(x) \Pi_{j=0}^{r+1} |x-c_j|^{\g_j}$.
Denote by $h_j$ the function
$$
h_j(x):= \frac{w(x)}{|x-c_j|^{\g_j}|x-c_{j+1}|^{\g_{j+1}}}\,,
$$
for $0\le j \le r$.
It is clear that there exist positive constants $m_j,M_j$
(depending just on $m,M,c_1, \dots , c_{r}, \g_1, \dots , \g_{r}$),
with $m_j\le h_j(x) \le M_j$ for every $x \in [c_j,c_{j+1}]$.

Hence, for $P \in \PP_n$, we have
$$
\begin{aligned}
\|P'\|_{L^2([c_j,c_{j+1}],w)}
& = \Big( \int_{c_j}^{c_{j+1}} |P'(x)|^2\, h_j(x) |x-c_j|^{\g_j}|x-c_{j+1}|^{\g_{j+1}} \,dx \Big)^{1/2}
\\
& \le \sqrt{M_j}\, \Big( \int_{c_j}^{c_{j+1}} |P'(x)|^2\,  |x-c_j|^{\g_j}|x-c_{j+1}|^{\g_{j+1}} \,dx \Big)^{1/2}
\\
& \le \sqrt{M_j}\; C(c_j, c_{j+1}, \g_j, \g_{j+1})\, n^2 \Big( \int_{c_j}^{c_{j+1}} |P(x)|^2 \, |x-c_j|^{\g_j}|x-c_{j+1}|^{\g_{j+1}} \,dx \Big)^{1/2}
\\
& \le \sqrt{M_j}\; C(c_j, c_{j+1}, \g_j, \g_{j+1})\, n^2 \Big( \int_{c_j}^{c_{j+1}} |P(x)|^2 \, |x-c_j|^{\g_j}|x-c_{j+1}|^{\g_{j+1}} \frac{h_j(x)}{m_j} \;dx \Big)^{1/2}
\\
& = \sqrt{\frac{M_j}{m_j}}\; C(c_j, c_{j+1}, \g_j, \g_{j+1})\, n^2\, \|P\|_{L^2([c_j,c_{j+1}],w)}
.
\end{aligned}
$$

\smallskip

Next, ``pasting" several times this last inequality in each subinterval $[c_{j},c_{j+1}]\subseteq [a,b]$, $0\leq j\leq r$,  we obtain
$$
\|P'\|_{L^2(w)} \le C_1(a,b,c_1, \dots , c_{r}, \g_1, \dots , \g_{r},m,M)\, n^2\, \|P\|_{L^2(w)}
$$
for every polynomial $P \in \PP_n$, with
$$
C_1(a,b,c_1, \dots , c_{r}, \g_1, \dots , \g_{r},m,M)
:= \max_{0\leq j\leq r} \sqrt{\frac{M_j}{m_j}}\; C(c_{j},c_{j+1},\g_{j},\g_{j+1}) .
$$
Hence, we obtain the case $(4)$ by applying \eqref{ec2}.

\medskip

Similarly, for the case $(5.1)$ we can write
$$
w(x)= H_1(x) \prod_{j=1}^{r} |x-c_j|^{\g_j} ,
$$
where $H_1(x):= h(x) e^{-x}$ satisfies $0< m\,e^{-c_r} \le H_1 \le M$ in $[0,c_r]$.
Then the case $(4)$ provides a constant $C_1$, which just depends on the appropriate parameters, with
\begin{equation} \label{equat1}
\|P'\|_{W^{k,2}([0,c_r],w,\l_1w,\dots,\l_k w)} \le
 C_1\, n^2\, \|P\|_{W^{k,2}([0,c_r],w,\l_1w,\dots,\l_k w)},
\end{equation}
for every $\l_1,\dots,\l_k \ge 0$ and every polynomial $P \in \PP_n$.

Proposition \ref{prop1} $(1)$ gives
$$
\|P'\|_{L^2(w_1)} \le C_\a \,n\, \|P\|_{L^2(w_1)},
$$
where $w_1(x):= x^\a e^{-x}$ in $[0,\infty)$, $\alpha>-1$ and  $P \in \PP_n$.
Hence, replacing $x$ by $x-c$,
we obtain with the same constant $C_\a$
$$
\|P'\|_{L^2([c,\infty),(x-c)^\a e^{c-x})} \le C_\a \,n\, \|P\|_{L^2([c,\infty),(x-c)^\a e^{c-x})},
$$
for every $c\ge 0$ and $P \in \PP_n$.
Now, if $w_2(x):=(x-c)^\a e^{-x}$, then the previous inequality implies
$$
\|P'\|_{L^2([c,\infty),w_2)} \le C_\a \, n\, \|P\|_{L^2([c,\infty),w_2)},
$$
for every $c\ge 0$ and $P \in \PP_n$, and \eqref{ec2} gives
$$
\|P'\|_{W^{k,2}([c,\infty),w_2,\l_1w_2,\dots,\l_k w_2)} \le
 C_\a\, n\, \|P\|_{W^{k,2}([c,\infty),w_2,\l_1w_2,\dots,\l_k w_2)},
$$
for every $c\ge 0$, $\l_1,\dots,\l_k \ge 0$ and $P \in \PP_n$.

We can write now
$$
w(x)= H_2(x) (x-c_r)^{\g_r} e^{-x},
$$
where $H_2(x):= h(x) \Pi_{j=1}^{r-1} |x-c_j|^{\g_j}$
and there exist constants $m_2,M_2$ with
$0< m_2 \le H_2 \le M_2$ in $[c_r,\infty)$, since $\sum_{j=1}^{r-1} \g_j=0$.
Thus,
\begin{equation} \label{equat2}
\|P'\|_{W^{k,2}([c_r,\infty),w,\l_1w,\dots,\l_k w)} \le
 C_{\g_r}\sqrt{\frac{M_2}{m_2}}\; n\, \|P\|_{W^{k,2}([c_r,\infty),w,\l_1w,\dots,\l_k w)},
\end{equation}
for every $\l_1,\dots,\l_k \ge 0$ and $P \in \PP_n$.

If we define $C_2:=\max\big\{C_1,\, C_{\g_r}\sqrt{M_2/m_2}\, \big\}$, then \eqref{equat1} and \eqref{equat2} give
$$
\|P'\|_{W^{k,2}(w,\l_1w,\dots,\l_k w)} \le
 C_2\, n^2\, \|P\|_{W^{k,2}(w,\l_1w,\dots,\l_k w)},
$$
for every $\l_1,\dots,\l_k \ge 0$ and $P \in \PP_n$.

\medskip

Let us prove now $(5.2)$.
Define $A:=1+c_r$.
We can write
$$
w(x)= H_3(x) \prod_{j=1}^{r} |x-c_j|^{\g_j} ,
$$
where $H_3(x):= h(x) e^{-x}$ satisfies $0< m\,e^{-A} \le H_3 \le M$ in $[0,A]$.
Then the case $(4)$ provides a constant $C_1$, which just depends on the appropriate parameters, with
\begin{equation} \label{equat221}
\|P'\|_{W^{k,2}([0,A],w,\l_1w,\dots,\l_k w)} \le
 C_1\, n^2\, \|P\|_{W^{k,2}([0,A],w,\l_1w,\dots,\l_k w)},
\end{equation}
for every $\l_1,\dots,\l_k \ge 0$ and every polynomial $P \in \PP_n$.

Proposition \ref{prop1} $(1)$ gives a constant $C_s$ with
$$
\|P'\|_{L^2(w_3)}^2 \le C_s^2 n^2\, \|P\|_{L^2(w_3)}^2,
$$
where $w_3(x):= x^s e^{-x}$, $s:= \sum_{j=1}^{r} \g_j > -1$ and $P \in \PP_n$.

We can write now
$$
w(x)= H_4(x) x^{s} e^{-x}= H_4(x) w_3(x),
$$
where $H_4(x):= h(x) x^{-s} \Pi_{j=1}^{r} |x-c_j|^{\g_j}$,
and there exist constants $m_4,M_4$ with
$0< m_4 \le H_4 \le M_4$ in $[A,\infty)$, since $s= \sum_{j=1}^{r} \g_j$.
Thus,
\begin{equation} \label{equat222}
\begin{aligned}
\|P'\|^2_{L^2([A,\infty),w)}
& \le M_4 \|P'\|_{L^2(w_3)}^2
\le C_s^2 n^2 M_4\, \|P\|_{L^2(w_3)}^2
\\
& \le C_s^2 n^2\frac{M_4}{m_4}\;\|P\|^2_{L^2([A,\infty),\,w)} + C_s^2 n^2 M_4\,\|P\|^2_{L^2([0,A],\,w_3)},
\end{aligned}
\end{equation}
for every $P \in \PP_n$.

\smallskip

Using Lupa\c{s}' inequality \cite{L} (see also \cite[p.594]{MMR}):
$$
\|P\|_{L^\infty([-1,1])}\leq \sqrt{\frac{\Gamma(n+\alpha+\beta+2)}{2^{\alpha+\beta+1}\Gamma(q+1)\Gamma(n+q'+1)}{n+q+1\choose n}}\sqrt{ \int_{-1}^1 |P(x)|^{2} (1-x)^{\alpha} (1+x)^{\beta}dx},
$$
for every $P \in \PP_n$,
where $q=\max(\alpha,\beta)\geq-1/2$ and $q'=\min(\alpha,\beta)$, we obtain that
$$2\|P\|^{2}_{L^\infty([-1,1])}\leq \frac{\Gamma(n+\alpha+\beta+2)\Gamma(n+q+2)}{2^{\alpha+\beta}\Gamma(q+1)\Gamma(q+2)\Gamma(n+1)\Gamma(n+q'+1)}
\int_{-1}^1 |P(x)|^{2} (1-x)^{\alpha} (1+x)^{\beta}dx.$$

Now, taking into account that
$$
\lim_{n\to\infty}\frac{\Gamma(n+x)}{\Gamma(n+y)n^{x-y}}=1, \quad x,y\in\RR,
$$
we get
$$
\frac{\Gamma(n+\alpha+\beta+2)\Gamma(n+q+2)}{\Gamma(n+1)\Gamma(n+q'+1)} \sim n^{\alpha+\beta+1} n^{q-q'+1}
= n^{\alpha+\beta+ |\alpha-\beta|+2}.
$$
Consequently, there exists a constant $k_1$, which just depends on $\a$ and $\b$, such that
$$\int_{-1}^1 |P(x)|^{2} dx
\leq 2\|P\|^{2}_{L^\infty([-1,1])}\leq k_{1}(\alpha,\beta) n^{v(\alpha,\beta)}\int_{-1}^1 |P(x)|^{2} (1-x)^{\alpha} (1+x)^{\beta}dx,$$
where $v(\alpha,\beta)= \alpha+\beta+ |\alpha-\beta|+2$ and $P \in \PP_n$.

Recall that $\max\{\g_j',\g_{j+1}'\} \ge -1/2$ for every $r_0-1 \le j \le r$ and
$$
b':= \max_{r_0-1\le j \le r} \big(\g_j'+\g_{j+1}'+|\g_j'-\g_{j+1}'| + 2 \big) .
$$

Therefore, a similar argument to the one in the proof of $(4)$ gives
$$
\int_{0}^{A} |P(x)|^2 \, dx
\le k_2\, n^{b'} \int_{0}^{A} |P(x)|^2 \prod_{j=r_0-1}^{r+1} |x-c_j|^{\g_j'}\, dx
= k_2\, n^{b'} \int_{0}^{A} |P(x)|^2 \prod_{j=r_0}^{r} |x-c_j|^{\g_j}\, dx ,
$$
for every polynomial $P \in \PP_n$ and some constant $k_2$ which just depends on $c_{r_0},\dots , c_r, \g_{r_0}, \dots , \g_r$.
Thus,
$$
\int_{0}^{A} |P(x)|^2 \, dx
\le k_3\, n^{b'} \int_{0}^{A} |P(x)|^2 \prod_{j=1}^{r} |x-c_j|^{\g_j}\, dx,
$$
for every polynomial $P \in \PP_n$ and some constant $k_3$ which just depends on $c_{1},\dots , c_r, \g_{1}, \dots , \g_r$.

Hence,
$$
\begin{aligned}
\|P\|^2_{L^2([0,A],w_3)} & = \int_{0}^{A} |P(x)|^2 x^s e^{-x}\, dx
\le A^s \int_{0}^{A} |P(x)|^2 \, dx
\\
& \le k_3\, n^{b'} A^s \int_{0}^{A} |P(x)|^2 \prod_{j=1}^{r} |x-c_j|^{\g_j}\, dx
\\
& \le \frac1m\,k_3\, n^{b'} A^s e^{A} \int_{0}^{A} |P(x)|^2 h(x) \prod_{j=1}^{r} |x-c_j|^{\g_j} e^{-x}\, dx
\\
& \le \frac1m\,k_3\, A^s e^{A} n^{b'} \|P\|^2_{L^2(w)},
\end{aligned}
$$
for every polynomial $P \in \PP_n$.

This inequality and \eqref{equat222} give
$$
\begin{aligned}
\|P'\|^2_{L^2([A,\infty),w)}
& \le \max \Big\{ C_s^2 \,\frac{M_4}{m_4}\,n^2, \,C_s^2\,\frac{M_4}{m}\,k_3\,A^s e^{A} n^{b'+2} \Big\} \,\|P\|^2_{L^2(w)}
\\
& \le k_4\, n^{b'+2} \,\|P\|^2_{L^2(w)},
\end{aligned}
$$
for every polynomial $P \in \PP_n$, where
$$
k_4:= \max \Big\{ C_s^2\,\frac{M_4}{m_4}\;, C_s^2\,\frac{M_4}{m}\,k_3\, A^s e^{A}\Big\} .
$$
Hence,
$$
\|P'\|^2_{W^{k,2}([A,\infty),\,w,\l_1w,\dots,\l_k w)} \le
k_4\, n^{b'+2} \, \|P\|^2_{W^{k,2}(w,\l_1w,\dots,\l_k w)},
$$
for every $\l_1,\dots,\l_k \ge 0$ and every polynomial $P \in \PP_n$, and \eqref{equat221} allows to deduce
$$
\|P'\|^2_{W^{k,2}(w,\l_1w,\dots,\l_k w)} \le
\big( C_1^2 n^4 + k_4\, n^{b'+2} \big) \|P\|^2_{W^{k,2}(w,\l_1w,\dots,\l_k w)},
$$
for every $\l_1,\dots,\l_k \ge 0$ and every polynomial $P \in \PP_n$.
If we define
$$
k_5:= \big( C_1^2 + k_4 \big)^{1/2} ,
$$
and we recall that
$$
a':= \max \Big\{ 2\,,\,\frac{b'+2}{2}\, \Big\} ,
$$
then
$$
\|P'\|_{W^{k,2}(w,\l_1w,\dots,\l_k w)} \le
k_5\, n^{a'} \,\|P\|_{W^{k,2}(w,\l_1w,\dots,\l_k w)},
$$
for every $\l_1,\dots,\l_k \ge 0$ and every polynomial $P \in \PP_n$.

\medskip

Finally, let us show $(6)$.
Define $B:=1+\max\{|c_1|,|c_r|\}$.
We can write
$$
w(x)= H_5(x) \prod_{j=1}^{r} |x-c_j|^{\g_j} ,
$$
where $H_5(x):= h(x) e^{-x^2}$ satisfies $0< m\,e^{-B^2} \le H_5 \le M$ in $[-B,B]$.
Then the case $(4)$ provides a constant $C_1$, which just depends on the appropriate parameters, with
\begin{equation} \label{equat221b}
\|P'\|_{W^{k,2}([-B,B],w,\l_1w,\dots,\l_k w)} \le
 C_1\, n^2\, \|P\|_{W^{k,2}([-B,B],w,\l_1w,\dots,\l_k w)},
\end{equation}
for every $\l_1,\dots,\l_k \ge 0$ and every polynomial $P \in \PP_n$.

Proposition \ref{prop1} $(2)$ gives
$$
\|P'\|_{L^2(w_4)}^2 \le 2 n\, \|P\|_{L^2(w_4)}^2,
$$
where $w_4(x):= |x|^\a e^{-x^2}$, $\alpha:= \sum_{j=1}^{r} \g_j \ge 0$ and  $P \in \PP_n$.

We can write now
$$
w(x)= H_6(x) |x|^{\a} e^{-x^2}= H_6(x) w_4(x),
$$
where $H_6(x):= h(x) |x|^{-\a} \Pi_{j=1}^{r} |x-c_j|^{\g_j}$,
and there exist constants $m_6,M_6$ with
$0< m_6 \le H_6 \le M_6$ in $(-\infty,-B]\cup [B,\infty)$, since $\alpha= \sum_{j=1}^{r} \g_j$.
Thus,
\begin{equation} \label{equat222b}
\begin{aligned}
\|P'\|^2_{L^2((-\infty,-B]\cup [B,\infty),w)}
& \le M_6 \|P'\|_{L^2(w_4)}^2
\le 2 n M_6\, \|P\|_{L^2(w_4)}^2
\\
& \le 2n\,\frac{M_6}{m_6}\;\|P\|^2_{L^2((-\infty,-B]\cup [B,\infty),\,w)} + 2n M_6\,\|P\|^2_{L^2([-B,B],\,w_4)},
\end{aligned}
\end{equation}
for every $P \in \PP_n$.

\smallskip

Since $\max\{\g_j,\g_{j+1}\} \ge -1/2$ for every $0 \le j \le r$
and $b:= \max_{0\le j \le r} \big(\g_j+\g_{j+1}+|\g_j-\g_{j+1}| + 2 \big)$,
the argument in the proof of $(5.2)$, using Lupa\c{s}' inequality, gives
$$
\int_{-B}^{B} |P(x)|^2 \, dx
\le k_6\, n^b \int_{-B}^{B} |P(x)|^2 \prod_{j=1}^{r} |x-c_j|^{\g_j}\, dx ,
$$
for every polynomial $P \in \PP_n$ and some constant $k_6$ which just depends on $c_1,\dots , c_r, \g_1, \dots , \g_r$.

Hence,
$$
\begin{aligned}
\|P\|^2_{L^2([-B,B],w_4)} & = \int_{-B}^{B} |P(x)|^2 |x|^\a e^{-x^2}\, dx
\le B^\a \int_{-B}^{B} |P(x)|^2 \, dx
\\
& \le k_6\, n^b B^\a \int_{-B}^{B} |P(x)|^2 \prod_{j=1}^{r} |x-c_j|^{\g_j}\, dx
\\
& \le \frac1m\,k_6\, n^b B^\a e^{B^2} \int_{-B}^{B} |P(x)|^2 h(x) \prod_{j=1}^{r} |x-c_j|^{\g_j} e^{-x^2}\, dx
\\
& \le \frac1m\,k_6\, n^b B^\a e^{B^2} \|P\|^2_{L^2(w)},
\end{aligned}
$$
for every polynomial $P \in \PP_n$.

This inequality and \eqref{equat222b} give
$$
\begin{aligned}
\|P'\|^2_{L^2((-\infty,-B]\cup [B,\infty),w)}
& \le \max \Big\{ 2\,\frac{M_6}{m_6}\,n,\, 2\,\frac{M_6}{m}\,k_6\, B^\a e^{B^2}n^{b+1} \Big\} \,\|P\|^2_{L^2(w)}
\\
& \le k_7\, n^{b+1} \,\|P\|^2_{L^2(w)},
\end{aligned}
$$
for every polynomial $P \in \PP_n$, where
$$
k_7:= \max \Big\{ 2\,\frac{M_6}{m_6}\;, 2\,\frac{M_6}{m}\,k_6\, B^\a e^{B^2}\Big\} .
$$
Hence,
$$
\|P'\|^2_{W^{k,2}((-\infty,-B]\cup [B,\infty),\,w,\l_1w,\dots,\l_k w)} \le
k_7\, n^{b+1} \, \|P\|^2_{W^{k,2}(w,\l_1w,\dots,\l_k w)},
$$
for every $\l_1,\dots,\l_k \ge 0$ and every polynomial $P \in \PP_n$, and \eqref{equat221b} allows to deduce
$$
\|P'\|^2_{W^{k,2}(w,\l_1w,\dots,\l_k w)} \le
\big( C_1^2 n^4 + k_7\, n^{b+1} \big) \|P\|^2_{W^{k,2}(w,\l_1w,\dots,\l_k w)},
$$
for every $\l_1,\dots,\l_k \ge 0$ and every polynomial $P \in \PP_n$.
If we define
$$
k_8:= \big( C_1^2 + k_7 \big)^{1/2} ,
$$
and we recall that
$$
a:= \max \Big\{ 2\,,\,\frac{b+1}{2}\, \Big\} ,
$$
then
$$
\|P'\|_{W^{k,2}(w,\l_1w,\dots,\l_k w)} \le
k_8\, n^a \,\|P\|_{W^{k,2}(w,\l_1w,\dots,\l_k w)},
$$
for every $\l_1,\dots,\l_k \ge 0$ and every polynomial $P \in \PP_n$.
\end{proof}

\section{Sobolev spaces with respect to measures}

\label{[Section-3]-Basic}

In this section we recall the definition of Sobolev spaces with respect to measures introduced in \cite{RARP1},
\cite{RARP2} and \cite{RARP3}.

\begin{definition}
\label{2} Given $1\le p < \infty$ and a set $A$ which is a union of intervals in $\RR$, we say that a weight $w$
defined in $A$ belongs to $B_p(A)$ if $w^{-1}\in L_{loc}^{1/(p-1)} (A)$ (if $p=1$, then $1/(p-1)=\infty$).
\end{definition}

It is possible to construct a similar theory with $p=\infty$. We refer to \cite{APRR}, \cite{PQRT1}, \cite{PQRT2} and
\cite{PQRT3} for the case $p=\infty$.

\smallskip

If $A=\RR$, then $B_p(\RR)$ contains, as a very particular case, the classical $A_p(\RR)$ weights appearing in Harmonic
Analysis. The classes $B_p(\O)$, with $\O\subseteq\RR^n$, have been used in other definitions of weighted Sobolev
spaces in $\RR^n$ in \cite{KO}.

\smallskip

We consider vector measures $\mu =(\mu_0, \dots , \mu_k)$ in the definition of our Sobolev space in $\RR$.
We assume that each $\mu_j$ is $\s$-finite; hence, by Radon-Nikodym's Theorem, we
have the decomposition $d\mu_j= d(\mu_j)_s + w_j ds$, where $(\mu_j)_s$ is
singular with respect to Lebesgue measure and $w_j$ is a non-negative Lebesgue measurable function.

In \cite{KO}, Kufner and Opic define the following sets:

\begin{definition}
\label{4} Let us consider $1\le p < \infty$ and a vector measure $\mu = (\mu_0, \dots , \mu_k)$. For $0\le j \le k$
we define the open set
$$
\O_j:=\big\{ x \in \RR\; : \; \exists \hbox{  an open neighbourhood  } V \hbox{  of  } x  \hbox{ with  } w_j\in
B_p(V)\big\}\,.
$$
\end{definition}

Note that we always have $w_j \in B_p(\O_j)$ for any $0\le j \le k$. In fact, $\O_j$ is the largest open set $U$ with
$w_j\in B_p(U)$. It is easy to check that  if $f^{(j)} \in L^p(\O_j,w_j),$ $1\le j \le k $, then $f^{(j)} \in
L^1_{loc}(\O_j)$ and, therefore, $f^{(j-1)} \in AC_{loc}(\O_j)$, i.e., $f^{(j-1)}$ is a locally absolutely continuous
function in $\O_j$.

Since the precise definition of Sobolev space requires some technical concepts (see Definition \ref{9}), we would like
to introduce here a heuristic definition of Sobolev space and an example which will help us to understand the technical
process that we will follow in order to reach Definition \ref{9}.

\begin{definition}
\label{heuristic} (Heuristic definition.) Let us consider $1\le p< \infty$ and a $p$-admissible vector measure
$\mu=(\mu_0,\dots,\mu_k)$ in $\RR$. We define the \emph{Sobolev space} $W^{k,p}(\mu)=W^{k,p}(\D,\mu)$, with $\D:=\cup_{j=0}^k \supp
(\mu_j)$, as the space of equivalence classes of
$$
\begin{aligned}
V^{k,p}(\mu):=V^{k,p}(\D,\mu):=\Big\{f: & \D\rightarrow\RR \; \;  :  \;\; \big\|f\big\|_{W^{k,p}(\D,\mu)}:=\Big(\sum_{j=0}^k
\big\|f^{(j)}\big\|^p_ {L^p(\D,\mu_j)} \Big)^{1/p} < \infty \,,
\\
& f^{(j)}\in AC_{loc} (\O_{j+1}\cup\cdots\cup\O_k) \hbox{ and } \, f^{(j)} \hbox{ satisfies }
\\
& \hbox{``pasting conditions" for } 0\le j < k \Big\}\,,
\end{aligned}
$$
with respect to the seminorm $\|\cdot \|_{W^{k,p}(\D,\mu)}$.
\end{definition}

\smallskip

These pasting conditions are natural: a function must be as regular as possible. In a first step, we check if the
functions and their derivatives are absolutely continuous up to the boundary (this fact holds in the following example),
and then we join the contiguous intervals:

\smallskip

\noindent {\bf Example.} $\mu_0:=\d_0$, $\mu_1:=0$, $d\mu_2:=\chi_{{}_{\scriptstyle [-1,0]}}(x) dx$ and
$d\mu_3:=\chi_{{}_{\scriptstyle [0,1]}}(x) dx$, where $\chi_{_{A}}$ denotes the characteristic function of the set $A$.

Since $\O_1=\emptyset$, $\O_2=(-1,0)$ and $\O_3=(0,1)$, $W^{3,p}(\mu)$ is the space of equivalence classes of
$$
\begin{aligned}
V^{3,p}(\mu) = \Big\{ f \ :  \ \|f\|_{W^{3,p}(\mu)} & < \infty \,, \, \text{ $f,f'$ satisfy ``pasting conditions",}
\\
& \ \, f,f',\in AC((-1,0)) \, \text{ and } \, f,f',f''\in AC((0,1)) \Big\}
\\
= \Big\{ f \ :  \ \|f\|_{W^{3,p}(\mu)} & < \infty \,, \, \text{ $f,f'$ satisfy ``pasting conditions",}
\\
& \ \, f,f',\in AC([-1,0]) \, \text{ and } \, f,f',f''\in AC([0,1]) \Big\}
\\
= \Big\{ f \ :  \ \|f\|_{W^{3,p}(\mu)} & < \infty \,, \ \, f,f'\in AC([-1,1]) \, \text{ and } \, f''\in AC([0,1])
\Big\}\,.
\end{aligned}
$$
In the current case, since $f$ and $f'$ are absolutely continuous in $[-1,0]$ and in $[0,1]$, we require that both are
absolutely continuous in $[-1,1]$.

\smallskip

These heuristic concepts can be formalized as follows:

\begin{definition}
Let us consider $1\le p<\infty$ and $\mu,\nu$ measures in $[a,b]$. We define
$$
\begin{aligned}
\L_{p,[a,b]}^+ (\mu, \nu) & := \sup_{a<x<b} \mu((a,x])
 \big\|(d\nu/ds)^{-1}\big\|_{L^{1/(p-1)}([x,b])}\,,
\\
\L_{p,[a,b]}^- (\mu, \nu) & := \sup_{a<x<b} \mu([x,b))
 \big\|(d\nu/ds)^{-1}\big\|_{L^{1/(p-1)}([a,x])}\,,
\end{aligned}
$$
where we use the convention $0\cdot \infty=0$.
\end{definition}

\smallskip

\noindent {\bf Muckenhoupt inequality.} (See \cite{Mu}, \cite[p.44]{M}, \cite[Theorem 3.1]{APRR}) {\it Let us consider $1\le p<\infty$ and
$\mu_0,\mu_1$ measures in $[a,b]$. Then:

$(1)$ There exists a real number $c$ such that
$$
\Big\|\int_{x}^{b} g(t)\,dt \Big\|_{L^p((a,b],\mu_0)}\le
 c\,
\big\|g\big\|_{L^p((a,b],\mu_1)}
$$
for any measurable function $g$ in $[a,b]$, if and only if $\L_{p,[a,b]}^+ (\mu_0, \mu_1)<\infty$.

$(2)$ There exists a positive constant $c$ such that
$$
\Big\|\int_{a}^x g(t)\,dt \Big\|_{L^p([a,b),\mu_0)}\le
 c\,
\big\|g\big\|_{L^p([a,b),\mu_1)}
$$
for any measurable function $g$ in $[a,b]$, if and only if $\L_{p,[a,b]}^- (\mu_0, \mu_1)<\infty$. }

\begin{definition}
Let us consider $1\le p<\infty$. A vector measure $\overline \mu=(\overline \mu_0,\dots,\overline \mu_k)$ is a right
completion of a vector measure $\mu=(\mu_0,\dots,\mu_k)$ in $\RR$ with respect to $a$ in a right neighborhood
$[a,b]$, if $\overline \mu_k=\mu_k$ in $[a,b]$, $\overline \mu_j=\mu_j$ in the complement of $(a,b]$ and
$$
\overline{\mu}_j= \mu_j+\tilde{\mu}_{j}\,,\qquad  \text{in } (a,b] \;\, \text{ for } 0\le j<k\,,
$$
where $\tilde \mu_{j}$ is any measure satisfying $\tilde \mu_j((a,b])<\infty$ and $\Lambda_{p, [a,b]}^+ (\tilde
\mu_j,\overline \mu_{j+1})<\infty$.
\end{definition}

Muckenhoupt inequality guarantees that if $f^{(j)}\in L^p(\mu_j)$ and $f^{(j+1)}\in L^p(\overline\mu_{j+1})$, then
$f^{(j)}\in L^p(\overline\mu_j)$ (see some examples of completions in \cite{RARP1} and \cite{APRR}).

\smallskip

\begin{obs}
We can define a left completion of $\mu$ with respect to $a$ in a similar way.
\end{obs}

\begin{definition}
\label{d:regular} For $1\le p< \infty$ and a vector measure $\mu$ in $\RR$, we say that a point $a$ is right
$j$-regular $($respectively, left $j$-regular$)$, if there exist a right completion $\overline\mu$ $($respectively,
left completion$)$ of $\mu$ in $[a,b]$ and $j<i\le k$ such that $\overline{w}_{i}\in B_p ([a,b])$ $($respectively,
$B_p([b,a]))$. Also, we say that a point $a\in \g$ is $j$-regular, if it is right and left $j$-regular.
\end{definition}

\begin{obs}

{\bf 1.} A point $a$ is right $j$-regular $($respectively, left $j$-regular$)$, if at least one of the following
properties holds:

\smallskip

{\rm (a)} \ There exist a right (respectively, left) neighborhood $[a,b]$ (respectively, $[b,a]$) and $j<i\le k$ such
that $w_{i} \in B_p ([a,b])$ $($respectively, $B_p([b,a]))$. Here we have chosen $\tilde w_j=0$.

{\rm (b)} \ There exist a right (respectively, left) neighborhood $[a,b]$ (respectively, $[b,a]$) and $j<i\le k$,
$\a>0$, $\d< (i-j)p-1$, such that $w_{i}(x) \ge \a \, |x-a|^\delta$, for almost every $x\in [a,b]$ $($respectively,
$[b,a])$. See Lemma 3.4 in \cite{RARP1}.

{\bf 2.} If $a$ is right $j$-regular (respectively, left), then it is also right $i$-regular (respectively, left) for
each $0\le i \le j$.
\end{obs}

\smallskip

When we use this definition we think of a point $\{t\}$ as the union of two half-points $\{t^+\}$ and $\{t^-\}$. With
this convention, each one of the following sets
$$
\begin{aligned}
(a,b) \cup (b,c) \cup \{b^+\} & = (a,b) \cup [b^+,c) \ne (a,c) \,,
\\
(a,b) \cup (b,c) \cup \{b^-\} & = (a,b^-] \cup (b,c) \ne (a,c) \,,
\end{aligned}
$$
has two connected components, and the set
$$
(a,b) \cup (b,c) \cup \{b^-\} \cup \{b^+\} = (a,b) \cup (b,c) \cup \{b\} = (a,c)
$$
is connected.

We use this convention in order to study the sets of continuity of functions: we want that if $f\in C(A)$ and $f\in
C(B)$, where $A$ and $B$ are union of intervals, then $f\in C(A\cup B)$. With the usual definition of continuity, if
$f\in C([a,b))\cap C([b,c])$ then we do not have $f\in C([a,c])$. Of course, we have $f\in C([a,c])$ if and only if
$f\in C([a,b^-])\cap C([b^+,c])$, where by definition, $C([b^+,c])=C([b,c])$ and $C([a,b^-])=C([a,b])$. This idea can
be formalized with a suitable topological space.

\smallskip

Let us introduce some more notation. We denote by $\O^{(j)}$ the set of $j$-regular points or half-points, i.e.,
$x\in\O^{(j)}$  if and only if $x$ is $j$-regular, we say that $x^+ \in \Omega^{(j)}$  if and only if $x$ is right
$j$-regular, and we say that $x^- \in \Omega^{(j)}$  if and only if $x$  is left $j$-regular. Obviously,
$\O^{(k)}=\emptyset$ and  $\Omega_{j+1}\cup \cdots \cup \Omega_k\subseteq \O^{(j)}$. Note that $\O^{(j)}$ depends on
$p$.

\smallskip

Intuitively, $\O^{(j)}$ is the set of ``good" points at the level $j$ for the vector weight $(w_0,\dots , w_k)$:
every function $f$ in the Sobolev space must verify that $f^{(j)}$ is continuous in $\O^{(j)}$.

\smallskip

Let us present now the class of measures that we use in the definition of Sobolev space.

\begin{definition}
\label{8} We say that the vector measure $\mu=(\mu_0,\dots,\mu_k)$ in $\RR$ is \emph{$p$-admissible} if $\mu_j$ is $\s$-finite and
$\mu_j^*(\RR\setminus \Omega^{(j)})=0$, for $1\le j < k$, and $\mu_k^* \equiv 0$, where $d\mu_j^* := d\mu_j - w_j
\chi_{_{\Omega_{j}}} \! dx$ and $\chi_{_{A}}$ denotes the characteristic function of the set $A$ (then $d\mu_k= w_k
\chi_{_{\Omega_{k}}} \! dx$).
\end{definition}

\begin{obs}
{\bf 1.} The hypothesis of $p$-admissibility is natural. It would not be reasonable to consider Dirac's deltas in
$\mu_j$ in the points where $f^{(j)}$ is not continuous.

{\bf 2.} Note that there is not any restriction on $\mu_0$.

{\bf 3.} Every absolutely continuous measure $w=(w_0, \dots, w_k)$ with $w_j=0$ a.e. in $\RR\setminus \Omega_j$ for
every $1\le j \le k$, is $p$-admissible (since then $\mu_j^*=0$). It is possible to find a weight $w$ which does not
satisfy this condition, but it is a hard task.

{\bf 4.} $(\mu_j)_s \le \mu_j^*$, and the equality usually holds.
\end{obs}

\begin{definition}
\label{9} Let us consider $1\le p< \infty$ and a $p$-admissible vector measure $\mu=(\mu_0,\dots,\mu_k)$ in $\RR$.
We define the \emph{Sobolev space} $W^{k,p}(\mu)=W^{k,p}(\D,\mu)$, with $\D:=\cup_{j=0}^k \supp
(\mu_j)$, as the space of equivalence classes of
\begin{eqnarray}
V^{k,p}(\D,\mu):=\Big\{f:\D\rightarrow\RR \ &:& \ f^{(j)}\in AC_{loc} (\O^{(j)}) \hbox{ for } 0\le j < k\hbox{ and }
\nonumber \\ &\,& \ \big\|f\big\|_{W^{k,p}(\D,\mu)}:=\Big(\sum_{j=0}^k \big\|f^{(j)}\big\|^p_ {L^p(\D,\mu_j)}
\Big)^{1/p}
 <\infty \Big\} \,, \nonumber
\end{eqnarray}
with respect to the seminorm $\|\cdot \|_{W^{k,p}(\D,\mu)}$.
\end{definition}

\section{Basic results on Sobolev spaces with respect to measures}

\label{[Section-4]-Dual}

This definition of Sobolev space is very technical, but it has interesting properties: we know explicitly how
are the functions in $W^{k,p}(\D,\mu)$ (this is not the case if we define the Sobolev space as the closure of some space of smooth functions);
if $\D$ is a compact set and $\mu$ is a finite measure, then in many cases,
$W^{k,p}(\D,\mu)$ is equal to the closure of the space of polynomials (see \cite[Theorem 6.1]{APRR}). Furthermore, we
have powerful tools in $W^{k,p}(\D,\mu)$ (see \cite{RARP1}, \cite{RARP2}, \cite{APRR} and \cite{RS}).

\smallskip

In \cite[Theorem 4.2]{RS} appears the following main result in the theory (in fact, this result in \cite{RS} holds for measures defined in any curve in the complex plane instead of $\RR$).

\begin{teo} \label{t:c1}
Let us  consider $1\le p < \infty$ and a $p$-admissible vector measure $\mu=(\mu_0,\dots,\mu_k)$. Then the Sobolev
space $W^{k,p}(\D,\mu)$ is a Banach space.
\end{teo}

We want to remark that the proof of Theorem \ref{t:c1} is very long and technical: the paper \cite{RARP1} is mainly
devoted to prove a weak version of Theorem \ref{t:c1}, and using this version, \cite{RS} provides a very technical proof of
the general case.

\smallskip

For each $1\le p<\infty$ and $p$-admissible vector measure $\mu$ in $\RR$, consider the Banach space $\prod_{j=0}^k
L^{p}(\D,\mu_j)$ with the norm
$$
\|f\|_{\prod_{j=0}^k L^{p}(\D,\mu_j)} = \Big( \sum_{j=0}^k \|f_j\|_{L^{p}(\mu_j)}^p \Big)^{1/p}
$$
for every $f=(f_0,f_1,\dots,f_k)\in \prod_{j=0}^k L^{p}(\D,\mu_j)$.

\begin{definition} For each $1\le p<\infty$ we denote by $q$ the conjugate or dual exponent of $p$, i.e., $1/p+1/q=1$.
Consider a $p$-admissible vector measure $\mu$. If $f\in \prod_{j=0}^k L^{p}(\D,\mu_j)$ and $g\in \prod_{j=0}^k
L^{q}(\D,\mu_j)$, we define the product $(f,g)$ as
$$
(f,g)=\sum_{j=0}^k \int_\D f_j \, g_j\, d\mu_j \,.
$$
Let us consider the projection $P:\,W^{k,p}(\D,\mu)\longrightarrow \prod_{j=0}^k L^{p}(\D,\mu_j)$, given by
$Pf=(f,f',\dots,f^{(k)})$.
\end{definition}

\begin{teo} \label{t:reflexive}
Let $1\le p<\infty$ and $\mu$ a $p$-admissible vector measure. Then $W^{k,p}(\D,\mu)$ is separable. Furthermore, if
$1< p<\infty$, then it is reflexive and uniformly convex.
\end{teo}

\begin{proof}
The map $P$ is an isometric embedding of $W^{k,p}(\D,\mu)$ onto $W:=P(W^{k,p}(\D,\mu))$. Then $W$ is a closed subspace
since $\prod_{j=0}^k L^{p}(\D,\mu_j)$ and $W^{k,p}(\D,\mu)$ are Banach spaces by Theorem \ref{t:c1}.

If $1\le p<\infty$, then each $L^{p}(\D,\mu_j)$ is separable; furthermore, if $1< p<\infty$, then it is reflexive and
uniformly convex. Then $\prod_{j=0}^k L^{p}(\D,\mu_j)$ is separable (and reflexive and uniformly convex if $1<
p<\infty$) by \cite[p.8]{A}.

Since $W$ is closed and $\prod_{j=0}^k L^{p}(\D,\mu_j)$ is separable, $W$ is separable; furthermore, if $1<
p<\infty$, then $W$ is reflexive and uniformly convex since $\prod_{j=0}^k L^{p}(\D,\mu_j)$ is reflexive and uniformly
convex (see \cite[p.7]{A} and \cite{H}). Since $W$ and $W^{k,p}(\D,\mu)$ are isometric, $W^{k,p}(\D,\mu)$ also has these
properties.
\end{proof}

\begin{teo}
\label{t:dual}
Let $1\le p<\infty$, $q$ the dual exponent of $p$, and $\mu$ a $p$-admissible vector measure. Consider the canonical
map $J:\,\prod_{j=0}^k L^{q}(\D,\mu_j) \longrightarrow (W^{k,p}(\D,\mu))'$ defined by $J(v)=(\cdot,v)$, i.e.,
$\big(J(v)\big)(f)=(Pf,v)$. Then giving any $T \in (W^{k,p}(\D,\mu))'$ there exists $v \in \prod_{j=0}^k
L^{q}(\D,\mu_j)$ with
\begin{equation} \label{eq:dual}
T=J(v) \qquad and \qquad \big\|T\big\|_{(W^{k,p}(\D,\mu))'} = \|v\|_{\prod_{j=0}^k L^{q}(\D,\mu_j)} .
\end{equation}
Furthermore, if $1< p<\infty$, then there exists a unique $v \in \prod_{j=0}^k L^{q}(\D,\mu_j)$ verifying \eqref{eq:dual}.
\end{teo}

\begin{proof}
First of all we will prove that $J$ is, in fact, a map $J:\,\prod_{j=0}^k L^{q}(\D,\mu_j) \longrightarrow
(W^{k,p}(\D,\mu))'$. Given $v \in \prod_{j=0}^k L^{q}(\D,\mu_j)$, consider $J(v)$.
Then continuous and discrete H\"older's inequalities give
$$
\begin{aligned}
\big|\big(J(v)\big)(f)\big| &
= \big|(Pf,v)\big|
\le \sum_{j=0}^k \int_\D \big| f^{(j)} \, v_j \big| \, d\mu_j \le \sum_{j=0}^k
\|f^{(j)}\|_{L^{p}(\mu_j)} \|v_j\|_{L^{q}(\mu_j)}
\\
& \le \Big( \sum_{j=0}^k \|f^{(j)}\|_{L^{p}(\mu_j)}^p \Big)^{1/p} \Big( \sum_{j=0}^k \|v_j\|_{L^{q}(\mu_j)}^q
\Big)^{1/q} =\|f\|_{W^{k,p}(\D,\mu)} \|v\|_{\prod_{j=0}^k L^{q}(\D,\mu_j)} .
\end{aligned}
$$
Hence,
\begin{equation} \label{eq:holder2}
\big\|J(v)\big\|_{(W^{k,p}(\D,\mu))'}
\le \|v\|_{\prod_{j=0}^k L^{q}(\D,\mu_j)} .
\end{equation}
Thus we have proved that $J:\,\prod_{j=0}^k L^{q}(\D,\mu_j) \longrightarrow (W^{k,p}(\D,\mu))'$. Let us prove now that
$J$ is onto.

Consider $T \in (W^{k,p}(\D,\mu))'$. The map $P$ is an isometric isomorphism of $W^{k,p}(\D,\mu)$ onto
$W:=P(W^{k,p}(\D,\mu))$. Then $T\circ P^{-1} \in W'$ and
$$
\big\|T\circ P^{-1}\big\|_{W'} = \big\|T \big\|_{\big(W^{k,p}(\D,\mu)\big)'}.
$$
Since $W$ is a subspace of $\prod_{j=0}^k L^{p}(\D,\mu_j)$, by Hahn-Banach
Theorem there exists
$$
T_0 \in \Big(\prod_{j=0}^k L^{p}(\D,\mu_j)\Big)' = \prod_{j=0}^k \big( L^{p}(\D,\mu_j)\big)' = \prod_{j=0}^k
L^{q}(\D,\mu_j) ,
$$
with
$$
T_0|_{W'}=T\circ P^{-1},
\qquad
\big\|T_0\big\|_{\big(\prod_{j=0}^k L^{p}(\D,\mu_j)\big)'}=\big\|T\circ P^{-1}\big\|_{W'}.
$$
Therefore, there exists $v \in \prod_{j=0}^k L^{q}(\D,\mu_j)$ with $v \simeq T_0$, i.e.,
$$
T_0(f)= (f,v), \qquad \forall f\in W,
$$
and
\begin{equation} \label{eq:holder3}
\|v\|_{\prod_{j=0}^k L^{q}(\D,\mu_j)}= \big\|T_0\big\|_{\big(\prod_{j=0}^k L^{p}(\D,\mu_j)\big)'}=\big\|T\circ
P^{-1}\big\|_{W'}
 = \big\|T \big\|_{\big(W^{k,p}(\D,\mu)\big)'}
\,.
\end{equation}

Hence,
$$
T(f)= T_0(Pf)= (Pf,v)= \big(J(v)\big)(f), \qquad \forall f\in W^{k,p}(\D,\mu),
$$
$T=J(v)$ and $J$ is onto.

\smallskip

If $1<p<\infty$, let us consider the set $U:=\{u \in \prod_{j=0}^k L^{q}(\D,\mu_j) \, : \,\, T=J(u)\}$. Note that \eqref{eq:holder2} and
\eqref{eq:holder3} give
$$
\begin{aligned}
\big\|T\big\|_{(W^{k,p}(\D,\mu))'}
& = \inf \Big\{ \|u\|_{\prod_{j=0}^k L^{q}(\D,\mu_j)} \, : \, \, u\in U \Big\}
\\ & =
\min \Big\{ \|u\|_{\prod_{j=0}^k L^{q}(\D,\mu_j)} \, : \, \, u\in U \Big\}.
\end{aligned}
$$
It is easy to check that $U$ is a closed convex set in $\prod_{j=0}^k L^{q}(\D,\mu_j)$.
Since $\prod_{j=0}^k L^{q}(\D,\mu_j)$ is uniformly convex, this minimum is attained at a unique $u_0\in U$ (see, e.g., \cite[p.22]{Ch}), and
\eqref{eq:holder3} gives $u_0=v$.
\end{proof}


\begin{thebibliography}{99}

\bibitem{A} R. A. Adams, Sobolev Spaces, Academic Press, New York, 1975.

\bibitem{APRR} V. Alvarez, D. Pestana, J. M. Rodr\'{\i}guez, E. Romera, Weighted Sobolev spaces on curves, {\it J.
    Approx. Theory} {\bf 119} (2002) 41--85.

\bibitem{ADK} A. I. Aptekarev, A. Draux, V. A. Kalyagin, On the asymptotics of sharp constants in Markov-Bernstein
    inequalities in integral metrics with classical weight, {\it Commun. Moscow Math. Soc.}
    {\bf 55} (2000), 163--165.

\bibitem{ADK2} A. I. Aptekarev, A. Draux, V. A. Kalyagin, D. Tulyakov, Asymptotics of sharp constants of Markov-Bernstein inequalities in integral norm with Jacobi weight,
{\it Proc. Amer. Math. Soc.} In press.

\bibitem{bun}  M. Bun, J. Thaler, Dual lower bounds for approximate degree and Markov-Bernstein inequalities, In Automata, languages, and programming. Part I, 303--314,
Lecture Notes in Comput. Sci., \textbf{7965}, Springer-Verlag, Heidelberg, 2013.

\bibitem{Ch} E. W. Cheney, Introduction to Approximation Theory, AMS Chelsea Publishing, Providence, R. I., 1982
    (Second Edition).

\bibitem{CPRR} E. Colorado, D. Pestana, J. M. Rodr\'{\i}guez, E. Romera,
Muckenhoupt inequality with three measures and Sobolev orthogonal polynomials,
{\it J. Math. Anal. Appl.} {\bf 407} (2013), 369--386.

\bibitem{Dor} P. D\"{o}rfler, New inequalities of Markov type, {\it SIAM J. Math. Anal.} \textbf{18} (1987), 490--494.

\bibitem{DK} A. Draux, V. A. Kalyagin, Markov-Bernstein inequalities for generalized Hermite weight, {\it East J.
    Approx.} {\bf 12} (2006), 1--24.

\bibitem{DMS} A. Draux, B. Moalla, M. Sadik,
Generalized qd algorithm and Markov-Bernstein inequalities for Jacobi weight,
{\it Numer. Algorithms} {\bf 51} (2009), 429--447.

\bibitem{Goe} P. Goetgheluck, On the Markov Inequality in $L^{p}$-Spaces, {\it J. Approx. Theory} \textbf{62} (1990), 197--205.

\bibitem{GM} A. Guessab,  G. V. Milovanovic, Weighted $L^{2}$- Analogues of Bernstein's Inequality and Classical Orthogonal Polynomials,  {\it J. Math. Anal. Appl.} {\bf 182} (1994), 244--249.

\bibitem{H} O. Hanner, On the uniform convexity of $L^p$ and $l^p$, {\it Arkiv Mat.} \textbf{3} (1956), 239--244.

\bibitem{HST} E. Hille, G. Szeg\H{o}, J. D. Tamarkin, On some generalization of a theorem of A. Markoff, {\it Duke Math. J.} \textbf{3} (1937), 729--739.


\bibitem{Kroo} A. Kro\'o, On the exact constant in the $L_{2}$ Markov inequality, {\it J. Approx. Theory} \textbf{151} (2008), 208--211.

\bibitem{KO} A. Kufner, B. Opic, How to define reasonably weighted Sobolev Spaces, {\it Comm. Math.
    Univ. Carol.} {\bf 25}(3) (1984), 537--554.

\bibitem{kwon} K. H. Kwon, D. W. Lee,  Markov-Bernstein type inequalities for polynomials,  {\it Bull. Korean Math.
    Soc.} {\bf 36} (1999), 63--78.

\bibitem{LPP} G. L\'opez Lagomasino, I. P\'erez Izquierdo, H. Pijeira, Asymptotic of extremal polynomials in the
complex plane, {\it J. Approx. Theory} {\bf 137} (2005), 226--237.

\bibitem{LP} G. L\'opez Lagomasino, H. Pijeira, Zero location and
$n$-th root asymptotics of Sobolev orthogonal polynomials, {\it J. Approx. Theory} {\bf 99} (1999), 30--43.

\bibitem{L} A. Lupa\c{s}, An inequality for polynomials, {\it Univ. Beograd Publ. Elektrotehn. Fak. Ser. Mat. Fiz.} {\bf 461-497} (1974), 241--243.

\bibitem{M} V. G. Maz'ja, Sobolev Spaces, Springer-Verlag, New York, 1985.

 \bibitem{nai} L. Milev, N. Naidenov, Exact Markov inequalities for the Hermite and
Laguerre weights, {\it J. Approx. Theory} {\bf 138} (2006), 87--96.

\bibitem{milo} G. V. Milovanovi\'{c}, Extremal problems and inequalities of Markov-Bernstein type for polynomials,  in Analytic and Geometric Inequalities and Applications, {\it  Mathematics and Its Applications}. T. M. Rassias and H. M. Srivastava Editors, {\bf 478} (1999),  245--264.

\bibitem{MMR} G. V. Milovanovi\'{c}, D. S. Mitrinovi\'{c}, Th. M. Rassias,
Topics in polynomials: extremal problems, inequalities, zeros.
World Scientific, Singapore, 1994.

\bibitem{mirsky} L. Mirsky,  An inequality of the Markov-Bernstein type for polynomials, {\it SIAM J. Math. Anal.} {\bf
    14} (1983), 1004--1008.

\bibitem{Mu} B. Muckenhoupt, Hardy's inequality with weights, {\it Studia Math.} {\bf 44} (1972), 31--38.

\bibitem{PQ2011} D. P\'erez, Y. Quintana, Some Markov-Bernstein type
 inequalities and certain class of Sobolev polynomials, \textit{J. Adv. Math. S.} \textbf{4} (2011) 85--100.

\bibitem{PQRT1} A. Portilla, Y. Quintana, J.M. Rodr\'{\i}guez, E. Tour\'{\i}s, Weierstrass' Theorem with weights, {\it
    J. Approx. Theory} {\bf 127} (2004), 83--107.

\bibitem{PQRT2} A. Portilla, Y. Quintana, J.M. Rodr\'{\i}guez, E. Tour\'{\i}s, Zero location and asymptotic behavior
    for extremal polynomials with non-diagonal Sobolev norms, {\it J. Approx. Theory} {\bf 162} (2010), 2225--2242.

\bibitem{PQRT3} A. Portilla, Y. Quintana, J.M. Rodr\'{\i}guez, E. Tour\'{\i}s, Concerning asymptotic behavior for
    extremal polynomials associated to non-diagonal Sobolev norms, {\it J. Function Spaces Appl.} {\bf 2013} (2013),
    Article ID 628031, 11 pages.

\bibitem{R} J. M. Rodr\'{\i}guez, The multiplication operator in
Sobolev spaces with respect to measures, {\it J. Approx. Theory}
{\bf 109} (2001), 157--197.

\bibitem{R2} J. M. Rodr\'{\i}guez,
A simple characterization of weighted Sobolev spaces with bounded multiplication operator,
{\it J. Approx. Theory} {\bf 153} (2008) 53--72.

\bibitem{R3} J. M. Rodr\'{\i}guez, Zeros of Sobolev orthogonal polynomials via Muckenhoupt inequality with three measures.
Submitted.

\bibitem{RARP1} J. M. Rodr\'{\i}guez, V. Alvarez, E. Romera, D. Pestana, Generalized weighted Sobolev spaces and
    applications to Sobolev orthogonal polynomials I, {\it Acta Appl. Math.} {\bf 80} (2004), 273--308.

\bibitem{RARP2} J. M. Rodr\'{\i}guez, V. Alvarez, E. Romera, D. Pestana, Generalized weighted Sobolev spaces and
    applications to Sobolev orthogonal polynomials II, {\it Approx. Theory and its Appl.} {\bf 18:2} (2002), 1--32.

\bibitem{RARP3} J. M. Rodr\'{\i}guez, V. Alvarez, E. Romera, D. Pestana, Generalized weighted Sobolev spaces and
    applications to Sobolev orthogonal polynomials: a survey, {\it Electr. Trans. Numer. Anal.} {\bf 24} (2006),
    88--93.

\bibitem{RS} J. M. Rodr\'{\i}guez, J. M. Sigarreta, Sobolev spaces with respect to measures in curves and zeros of
    Sobolev orthogonal polynomials, {\it Acta Appl. Math.} {\bf 104} (2008), 325--353.

\bibitem{schmidt} E. Schmidt,  \"{U}ber die nebst ihren Ableitungen orthogonalen Polynomensysteme und das
    zugeh\"{o}rige Extremum, {\it Math. Ann.} {\bf 119} (1944), 165--204.

\end{thebibliography}
\end{document}